\documentclass[12pt,fleqn]{amsart}

\usepackage{amssymb}

\usepackage{color}
\usepackage{enumerate}
\newtheorem{theorem}{Theorem}[section]
\newtheorem{lemma}[theorem]{Lemma}
\newtheorem{proposition}[theorem]{Proposition}
\newtheorem{problem}[theorem]{Problem}
\newtheorem{corollary}[theorem]{Corollary}

\theoremstyle{definition}
\newtheorem{definition}[theorem]{Definition}

\newtheorem{xca}[theorem]{Construction}

\newtheorem{axiom}[theorem]{Axiom}
\theoremstyle{remark}

\numberwithin{equation}{section}

\begin{document}

\title[Boolean algebras and s.p.\ measures]{On Boolean algebras with strictly positive measures}
%Reflecting properties of compact spaces\\ in small continuous images}

%    author two information
\author[M.\ Magidor]{Menachem Magidor}
\address{Hebrew University of Jerusalem}
\email{mensara@savion.huji.ac.il}

\author[G.\ Plebanek]{Grzegorz Plebanek}
\address{Instytut Matematyczny, Uniwersytet Wroc\l awski}
\email{grzes@math.uni.wroc.pl}

\thanks{
This research started during the conference {\em Set-theoretic methods in topology and analysis} (B\c{e}dlewo, September 2017) organized on the occasion of Professor Aleksander
B{\l}aszczyk's 70th birthday.
The second author is grateful to Boban Veli\v{c}kovi\'c for a valuable discussion on the subject  during
{\em The 6th European Set Theory Conference} (Budapest, July 2017) and for bringing \cite{FV06} to our attention.  }

\newcommand{\spm}{\protect{\sf SPM}}
\newcommand{\con}{\mathfrak c}

\newcommand{\fA}{\mathfrak A}
\begin{abstract}
We investigate reflection-type  problems on the class $\spm$, of Boolean algebras carrying strictly positive finitely additive measures.
We show, in particular, that in the constructible universe there is a Boolean algebra $\fA$ which is not  in $\spm$
but every subalgebra of $\fA$ of cardinality $\con$ admits a strictly positive measure.
This result is essentially due to Farah and Veli\v{c}kovi\'c \cite{FV06}.
\end{abstract}

\maketitle

%%%%%%%%%%%%%%%%%%%%%%%%%%%%%%%%%%%%%%%%%LOCAL SHORTENINGS
%%%%%%%%%%%%%%NUMBERS

\newcommand{\eps}{\varepsilon}
%%%%%%%%%%%%%%%%%%%%FRAK FAMILIES

\newcommand{\alg}{\mathfrak A}
\newcommand{\algb}{\mathfrak B}
\newcommand{\algc}{\mathfrak C}
\newcommand{\fB}{\mathfrak B}
\newcommand{\fC}{\mathfrak C}
\newcommand{\cP}{\mathcal P}
\newcommand{\ma}{\mathfrak M}
\newcommand{\pa}{\mathfrak P}
%%%%%%%%%%%%%%%%%%%%SCRIPT FAMILIES
\newcommand{\BB}{\protect{\mathcal B}}
\newcommand{\AAA}{\mathcal A}
\newcommand{\CC}{{\mathcal C}}
\newcommand{\cE}{{\mathcal E}}
\newcommand{\cF}{{\mathcal F}}
\newcommand{\FF}{{\mathcal F}}
\newcommand{\GG}{{\mathcal G}}
\newcommand{\LL}{{\mathcal L}}
\newcommand{\NN}{{\mathcal N}}
\newcommand{\UU}{{\mathcal U}}
\newcommand{\VV}{{\mathcal V}}
\newcommand{\HH}{{\mathcal H}}
\newcommand{\DD}{{\mathcal D}}
\newcommand{\RR}{\protect{\mathcal R}}
\newcommand{\ide}{\mathcal N}
%%%%%%%%%%%%%%%%%%%%%%%SYMBOLS
\newcommand{\btu}{\bigtriangleup}
\newcommand{\hra}{\hookrightarrow}
\newcommand{\ve}{\vee}
\newcommand{\we}{\cdot}
\newcommand{\de}{\protect{\rm{\; d}}}
\newcommand{\er}{\mathbb R}
\newcommand{\qu}{\mathbb Q}
\newcommand{\supp}{{\rm supp} }
\newcommand{\card}{{\rm card} }
\newcommand{\wn}{{\rm int} }
\newcommand{\ult}{{\rm ult}}
\newcommand{\vf}{\varphi}
\newcommand{\osc}{{\rm osc}}
\newcommand{\cov}{{\rm cov}}
\newcommand{\cf}{{\rm cf}}
\newcommand{\ol}{\overline}
\newcommand{\me}{\protect{\bf v}}
\newcommand{\ex}{\protect{\bf x}}
\newcommand{\stevo}{Todor\v{c}evi\'c}
\newcommand{\cc}{\protect{\mathfrak C}}
\newcommand{\scc}{\protect{\mathfrak C^*}}
\newcommand{\lra}{\longrightarrow}
\newcommand{\sm}{\setminus}
\newcommand{\uhr}{\upharpoonright}

\newcommand{\sub}{\subseteq}
\newcommand{\ms}{$(M^*)$}
\newcommand{\m}{$(M)$}
\newcommand{\MA}{$\protect{{\mathsf MA}}(\omega_1)$}
\newcommand{\clop}{\protect{\rm Clop} }
\newcommand{\fX}{\mathfrak X}
\newcommand{\fY}{\mathfrak Y}
\newcommand{\fZ}{\mathfrak Z}
\newcommand{\fr}{\kappa}

\newcommand{\spms}{\sf SPM^*}
\newcommand{\nmam}{NMA$^-$}
%%%%%%%%%%%%%%%%%%%%%%%%%%%%%%%%%%%%%%%%%%%%%%%%%%%%%%%%%%%
\section{Introduction}

%There is a significant amount of research related to properties of structures that reflect in substructures of smaller cardinality, see
%e.g.\  Bagaria, Magidor, Sakai \cite{BMS}, Koszmider \cite{Ko99,Ko05}, Fuchino and Rinot \cite{FR11}, Tall \cite{Ta07}.
%Reflection phenomena in topology are usually studied according to  the following pattern:

Given a Boolean algebra $\fA$, we write $\fA\in\spm$ to denote that $\fA$ carries a strictly positive measure, that is, there is
a finitely additive function $\mu:\fA\to\er_+$ such that $\mu(a)>0$ for every  $a\in\fA^+=\fA\sm\{0\}$.

It is easy to check that every $\sigma$-centred algebra $\fA$ is in \spm.
Let us recall that there is a combinatorial characterization of algebras from the class \spm\ due to Kelley \cite{Ke59}. Namely,
$\fA\in\spm$ if and only if there is a decomposition $\fA^+=\bigcup_n \cE_n$, where every family $\cE_n$ has the positive intersection number.
By definition, the intersection number of $\cE\sub\fA$ is $\ge \eps$ if for every $n$,  every sequence $a_1,\ldots, a_n\in\cE$ contains a subsequence of length $\ge\eps\cdot n$
with nonzero joint; cf.\ \cite{MN80} and \cite{To00}.

We consider here the problems of the following type.

\begin{problem}\label{i:1}
Let $\fr$ be a cardinal number.
Suppose that a Boolean  algebra $\fA$ has the property that $\fB\in\spm$ for every subalgebra $\fB$ of $\fA$ of cardinality $<\fr$.
Must $\fA$ itself be in \spm?
\end{problem}

The answer to the above question is clearly negative for $\fr=\omega_1$ since every countable algebra is in \spm. Note also
that Problem \ref{i:1} has a negative consistent answer for $\fr=\omega_2$. Indeed, assume that $\con=\omega_2$ and that Martin's axiom MA$(\omega_1)$ holds;
 let $\fA$ be the Gaifman algebra \cite{Ga64}, that is, $\fA$ is a $ccc$ algebra not carrying a strictly positive measures. Then every subalgebra $\fB$ of $\fA$ of size $\le\omega_1$ is
$\sigma$-centred by  Martin's axiom.

We shall discuss the above problem for $\kappa=\con^+$, which seems to be the most natural question.
It turns out, that   the positive answer to Problem \ref{i:1},  even for $\kappa=\con$, is a consequence of the Normal Measure Axiom, see Section \ref{lc}.
This may be seen by a simple adaptation of an argument due to Fremlin \cite{Fr93}.
On the other hand,  \ref{i:1} has a negative answer for $\kappa=\con^+$
in the constructible universe: in Section we construct a Boolean algebra $\fA\notin\spm$ such that $|\fA|=\con^+$ and
$\fB\in\spm$ for every subalgebra $\fB$ of $\fA$ with $|\fB|\le\con$.
Our proof  is in fact  a variant of an argument leading to the main result from \cite{FV06}.

 The construction mentioned above is based on the existence of a stationary set $S$ in $\omega_2$, consisting of ordinals of countable cofinality, that does not reflect, i.e.\ $S\cap \xi$ is stationary in $\xi$ for no limit ordinal $\xi<\omega_2$.
Various aspects of (non)reflecting stationary sets were discussed in  \cite{Ma82} and \cite{BMS} and found several applications in topology and functional analysis, see e.g.\ \cite{Fl80}, \cite{Ta07}  and \cite{MP17}.

In the light of our results on  Problem \ref{i:1} presented in this note, the following question seems to be quite interesting.

\begin{problem}
Is it  consistent with GCH that every algebra $\fA\notin\spm$ contains a subalgebra $\fB\notin\spm$ of cardinality $\le\con$?
\end{problem}

\section{Assuming large cardinals}\label{lc}

We discuss here an essentially known  partial solution to Problem \ref{i:1} with $\fr=\con$.

\begin{axiom}\label{lc:1}
We write \nmam\  for the following assertion:
\medskip

{\em For every set $X$ there is a countably additive probability measure defined on $\lambda$ on $\cP([X]^{<\con})$ such that
\[\mu\left( \{A\in [X]^{<\con}: x\in A\}\right)=1 \mbox{ for every } x\in X.\]}
 \end{axiom}

Note that the assertion of \nmam\  holds trivially for all sets $X$ with $|X|<\con$.
\nmam\ is formally weaker than NMA, the normal measure axiom, introduced by Fleissner \cite{Fl91}. The full version of NMA requires that the measure $\mu$ in question
is $\con$-additive and normal. Recall that, by a result due to Prikry,  the consistency of NMA is implied by the existence of a supercompact cardinal, see \cite{Fl91}.

The proof of the following theorem is a straightforward  adaptation of the argument from  Fremlin \cite[8R]{Fr93}.

\begin{theorem}\label{lc:2}
Assume \nmam.  Suppose that  $\fA$ is a Boolean algebra such that $\fB\in\spm$ for every subalgebra $\fB$ of $\fA$ with $|\fB|<\con$.
Then $\fA\in\spm$.
\end{theorem}

\begin{proof}
We use \nmam to find a countably additive probability measure $\lambda$ on $\cP([\fA]^{<\con})$ such that
$\lambda(\{B\in [\fA]^{<\con}: a\in B\})=1$ for every $a\in\fA$.

Given $B\sub\fA$ with $|B|<\con$, the subalgebra $\fB\sub\fA$ generated by $B$ is also of cardinality $<\con$. By the assumption on $\fA$,
there is a strictly positive finitely additive measure $\mu_B$ on $\fB$.
We define $\mu$ on $\fA$ by the formula
\[\mu(a)=\int_{[\fA]^{<\con}}\mu_B(a)\; {\rm d}\lambda(B).\]
Note that, for every $a$ in $\fA$, we have $a\in B$ for $\lambda$-almost all $B$; hence, the above integral is well-defined.
Likewise, for disjoint $a_1,a_2\in\fA$ we have $\mu_B(a_1\vee a_2)=\mu_B(a_1)+\mu_B(a_2)$ $\lambda$-almost everywhere.
Therefore, by the linearity of the integral, $\mu(a_1\vee a_2)=\mu(a_1)+\mu(a_2)$, so $\mu$ is a finitely additive probability measure
on $\fA$.

Finally, if  $a\in\fA^+$ then the function $B\to \mu_B(a)$ is positive  almost everywhere. Consequently, the integral of  such a function with respect to
a countably additive measure is positive.
\end{proof}

A compact (Hausdorff) space $K$ carries a strictly positive measure if there is a regular Borel probability measure $\mu$ such that
$\mu(V)>0$ for every nonempty open set $V\sub K$. Let us write $\spms$ for the class of compacta admitting a strictly positive measure.

Note that for a compact zero-dimensional space $K$, $K\in\spms$ if and only if the algebra $\clop(K)$, of closed-and-open subsets of $K$, is in $\spm$. Indeed, if $\mu$ is a Borel measure on $K$ then the restriction of $\mu$ to $\clop(K)$  is strictly positive (finitely additive) measure on a Boolean algebra. Conversely, given finitely additive strictly positive $\mu$ on $\clop(K)$, there is a unique extension $\widetilde{\mu}$ of $\mu$ to a regular measure on $Bor(K)$; clearly $\widetilde{ \mu}(V)>0$ for every nonempty open set.

The class $\spms$ is discussed in \cite[Chapter 6]{CN82}.  Recall that $\spms$ contains all metrizable compacta, is closed under taking arbitrary products and continuous images. We show below that Theorem \ref{lc:2} yields a reflection-type result for the class
$\spms$, which is in the spirit of properties considered in Tkachuk \cite{Tk12} and Tkachenko \& Tkachuk \cite{TT15}.

Recall first that if $g:K\to L$ is a continuous mapping between topological spaces and $\mu$ is a Borel measure on $K$, then
the image measure $g[\mu]$ is a Borel measure on $L$ defined by the formula $g[\mu](B)=\mu(g^{-1}[B])$ for  $B\in Bor(L)$.
It is well-known that, in the case $K$ is compact and $g$ is surjective, for every Borel measure $\nu$ on $L$ there is
a Borel measure $\mu$ on $K$ such that $g[\mu]=\nu$.

Below we denote by $w(\cdot )$ the weight of a topological space.

\begin{theorem}\label{lc:3}
Assume \nmam. Suppose that $K$ is a compact space such that $L\in\spms$ for every continuous image $L$ of $K$ with $w(L)<\con$. Then  $K\in \spms$.
\end{theorem}

\begin{proof}
Let us consider the Gleason space $G$ of $K$, i.e.\  an essentially unique extremally  disconnected compact space $G$ which
can be mapped onto $K$ by an irreducible mapping $r:G\to K$.

Suppose that $K\notin\spms$; then $G$ is not in $\spms$ either, see Corollary  6.3 in \cite{CN82}.
The space $G$ is zero-dimensional  so $\fA=\clop(G)\notin\spm$. By Theorem \ref{lc:2}, $\fA$ must contain a subalgebra
$\fB$ of size $<\con$ such that $\fB\notin\spm$.

For every $B\in\fB^+$, the set $r[G\sm B]$ is a proper closed subset of $K$
(since $r$ is irreducible). Pick a continuous
non-zero function  $f_B:K\to [0,1]$ which vanishes on  $r[G\sm B]$.
Let $g:K\to [0,1]^\fB$ be the diagonal map defined by $g(x)(B)=f_B(x)$ for $x\in K$ and $B\in\fB$.

Note that for the space $L=g[K]$, we have
\[ w(L)\le w([0,1]^\fB)=|\fB|<\con,\]
so to complete the proof it is enough to check  that $L\notin\spms$.

Indeed, take any probability Borel measure $\nu$ on $L$. Then, by the fact mentioned prior to the theorem,
there is a Borel probability measure $\mu$ on $G$, such that $g\circ r[\mu]=\nu$. By the choice of $\fB$,
$\mu(B)=0$ for some $B\in\fB^+$.
 Put $U=\{y\in L: y(B)>0\}$; then $U$ is open in $L$ and nonempty,
 for taking $x\in K$ such that $f_B(x)>0$ we get $y=g(x)\in U$.
Now  \[(g\circ r)^{-1}(U)=r^{-1}\left[\{x\in K: f_B(x)>0\}\right] \sub B,\]
and therefore  $\nu(U)=\mu(B)=0$. This shows that $\nu$ is not strictly positive, and
we are done.
\end{proof}

\section{Extensions of measures}\label{em}
Throughout this section, by a measure we mean a probability finitely additive measure.
We collect here some standard observations concerning extensions of  measures on Boolean
algebras. Then we prove Proposition \ref{em:4} that will be applied for the construction carried out in the next section.

For simplicity, consider an algebra $\fA$ of subsets of  some  set $X$, and a finitely additive
$\mu$ on $\fA$. For any $Z\sub X$ we write
\[ \mu^*(Z)=\inf\{\mu(A): A\in\fA, A\supseteq Z\}, \quad \mu_*(Z)=\sup\{\mu(A): A\in\fA, A\subseteq Z\}.\]
Note that $\fA(Z)$, the algebra generated by $\fA\cup\{Z\}$, is equal to the family
of all sets of the form $(A\cap Z)\cup (B\cap Z^c)$, where $A,B\in\fA$.

\begin{theorem}[{\L}o\'s and Marczewski \cite{LM49}]\label{em:1}
Let $\mu$ be a measure on an algebra $\fA$ of subsets of $X$. For every $Z\sub X$
the formulas
\[ \overline{\mu}\big( (A\cap Z)\cup (B\cap Z^c)\big)=\mu^*(A\cap X)+\mu_*(B\cap Z^c),\]
\[ \underline{\mu}\big( (A\cap Z)\cup (B\cap Z^c)\big)=\mu_*(A\cap X)+\mu^*(B\cap Z^c),\]
define extensions of $\mu$ to  measures $\overline{\mu}, \underline{\mu}$ on
$\fA(Z)$.
\end{theorem}

\begin{corollary}\label{em:2}
Suppose that $\mu$ is a strictly positive measure on an algebra $\fA$ of subsets of $X$. Given $Z\sub X$, suppose that
the sets $Z^0=Z$ and $Z^1=X\sm Z$ satisfy the condition
\[ \mu^*(A\cap Z^i)>0\mbox{ whenever } A\in \fA\mbox{ and } A\cap Z^i \neq\emptyset.\]
Then $\mu$ admits an extension to a  strictly positive measure on $\fA(Z)$.
\end{corollary}

\begin{proof}
Take the measure  $\nu=1/2(\overline{\mu}+\underline{\mu})$, where $\overline{\mu},\underline{\mu}$ are as in Theorem \ref{em:1};
clearly, $\nu$ is also an extension of $\mu$ to a measure on $\fA(Z)$.
It follows immediately from the assumption that $\nu$ is strictly positive.
\end{proof}

Consider now the space of the form $X=2^\kappa$. For any $\alpha<\kappa$ and $k\in\{0,1\}$ we put
\[C^k_\alpha=\{x\in 2^\kappa: x_\alpha=k\}.\]
A set $A\sub X$ is determined by coordinates in $I\sub\kappa$ if
$A=\pi_I^{-1}\pi_I[A]$, where $\pi_I$ is the projection $2^\kappa\to 2^I$. This is equivalent to saying that whenever $x\in A$ and $y\in X$ agrees with $x$ on $I$ then $y\in A$.

Given $I\sub\kappa$, we write $\CC[I]$ for the family of sets determined by coordinates in some finite subset of $I$ so that
$\CC[I]$ is the family of clopen subsets of the Cantor cube $2^\kappa$ determined by coordinates in $I$.
We denote by $Ba[I]$ the $\sigma$-algebra of subsets of $2^\kappa$ generated by $\CC[I]$.
Note that every set $B\in Ba[I]$ is determined by coordinates in some countable subset of $I$.
(Our notation is related to the fact that $Ba[\kappa]$ is the Baire $\sigma$-algebra of $2^\kappa$, the smallest one making all the continuous functions on $2^\kappa$ continuous.)

For a limit ordinal $\xi$ we denote
\[ Ba^<[\xi]=\bigcup_{\beta<\xi} Ba[\beta].\]
Note that if $\cf(\xi)>\omega$ then $Ba^<[\xi]=Ba[\xi]$.

\begin{lemma}\label{em:3}
Let $\fA$ be an algebra contained in some $Ba[I]$ and let $Z\in Ba[\kappa\sm I]$.
Then every strictly positive measure on $\fA$ can be extended to a strictly positive measure on $\fA(Z)$.
\end{lemma}

\begin{proof}
Clearly, we can assume that $Z\neq\emptyset$ and $Z\neq X$.

Let $A,B\in\fA$ and suppose that $A\cap Z\sub B$. Then $A\sm B\sub X\sm Z$, which implies
$A\sub B$ since $A\sm B$ is determined by coordinates in $I$.

Let $\mu$ be strictly positive on $\fA$.
The above remark shows that $\mu^*(A\cap Z)=\mu(A)$ for every $A\in\fA$.
We can apply the same argument to $X\sm Z$.
Hence, we finish the proof applying  Lemma \ref{em:2}.
\end{proof}

\begin{proposition}\label{em:4}
Let $\langle I_k: k\in\omega\rangle $ be a strictly increasing sequence of subsets of $\kappa$.
Suppose we are given

\begin{enumerate}[(i)]
\item an increasing  sequence of algebras $\fA_k$ such that
$\CC[I_k]\sub \fA_k\sub Ba[I_k]$ for every $k$;
\item some probability measure $\nu$ defined on $\fA=\bigcup_k\fA_k$.
\end{enumerate}

Then there is a set $Z\in Ba[\bigcup_k I_k]$ such that

\begin{enumerate}[(a)]
\item for every $k$, if $\mu$ is a strictly positive measure on $\fA_k$ then $\mu$ extends to a strictly positive
measure on $\fA_k(Z)$;
\item $\nu$ does not extend to a strictly positive measure on $\fA(Z)$.
\end{enumerate}
\end{proposition}

\begin{proof}
For every $k\ge 1$ pick $\alpha_k\in I_k\sm I_{k-1}$. Then define $j_k\in \{0,1\}$ inductively
so that
\[\mbox{writting } Z_k=\bigcap_{n\le k} C_{\alpha_n}^{j_n}\mbox{ we have } \nu(Z_{k+1})\le (1/2)\nu (Z_k).\]
We shall check that the set $Z=\bigcap_k Z_k$ is as required.

Clause (a) follows from Lemma \ref{em:3} and the fact that $\fA_k(Z)=\fA_k(Y_k)$, where $Y_k=\bigcap_{n>k} Z_n$ is determined by coordinates in $\kappa\sm I_k$.

To check (b) notice that, since $\nu(Z_k)\to 0$, we have $\nu^*(Z)=0$. Therefore $\widetilde{\nu}(Z)\le \nu^*(Z)=0$ whenever $\widetilde{\nu}$ extends $\nu$ to a measure on $\fA(Z)$.
\end{proof}

\section{A counterexample in $V=L$}\label{v=l}

Let $\gamma$ be a limit ordinal. Recall that  set $F\sub\gamma$ is said to be {\em closed}  if it is closed in the interval topology defined on ordinals smaller that $\gamma$.
Such a set $F$ is unbounded in $\gamma$ if for every $\beta<\gamma$ there is $\alpha\in F$ such that $\beta<\alpha$. A set $S\sub \gamma$ is {\em stationary}
if $S\cap F\neq\emptyset$ for every closed and unbounded $F\sub\gamma$.

It is not difficult to check that the set $S_\omega=\{\alpha<\omega_2: \cf(\alpha)=\omega\}$ is stationary in $\omega_2$. However, such a set reflects in the sense
that, for instance, $S_\omega\cap\omega_1$ is stationary in $\omega_1$.
We shall work assuming the following.

\begin{axiom}\label{ax1}
 There is a stationary set $S\sub\omega_2$ such that

\begin{enumerate}[(a)] \label{st:1}
\item $\cf(\alpha)=\omega$ for every $\alpha\in S$;
\item $S\cap\beta$ is not stationary in $\beta$ for every $\beta<\omega_2$ with $\cf(\beta)=\omega_1$.\label{st:1b}
\end{enumerate}
\end{axiom}

Basic information on \ref{ax1} can be found in Jech \cite{Jech}; recall that \ref{ax1} follows from Jensen's principle $\square_{\omega_1}$ (\cite{Jech}, Lemma 23.6)
and hence it holds  in the constructible universe (\cite{Jech}, Theorem 27.1).

Below we use the notation from the previous section. In particular, for $\xi<\omega_2$ we denote by $Ba[\xi]$ the family of Baire subsets of $2^{\omega_2}$ determined by coordinates in $\{\alpha: \alpha<\xi\}$.
Note that $Ba[\xi]$ has cardinality $\le\con$ for every $\xi<\omega_2$.

%For the construction given below we assume $\omega_1=2^\omega, \omega_2=2^{\omega_1} $ and \ref{ax1}.
\begin{axiom}\label{ax2}
$\omega_1=2^\omega ,\omega_2=2^{\omega_1}. $
\end{axiom}

\begin{definition} Let $\kappa$ be a regular cardinal, $S\subseteq\kappa$ a stationary subset of $\kappa$. The principle $\lozenge_S$ (introduced by Jensen ) states that there is a sequence $\langle D_\alpha : \alpha \in S\rangle$ such that for every $\alpha\in S$ we have $D_\alpha\subseteq \alpha $ and for every $D\subseteq \kappa$ the set
\[ \{ \alpha\in S : D\cap\alpha=S_\alpha \}\]
is stationary in $\kappa$.
\end{definition}

\begin{theorem}[Shelah \label{shelah} \cite{Sh10}]
Assume that

\begin{enumerate}[--]
\item $\kappa$ is a regular uncountable cardinal such that $2^\kappa=\kappa^+$;
 \item $S$ be stationary subset of $\kappa^+$ such that
$\cf(\alpha)\neq\kappa$ for $\alpha\in S$.
\end{enumerate}

Then $\lozenge_S$ holds.
\end{theorem}

 Note that if $2^\omega=\omega_1$ then for every limit ordinal $\xi$ with $\omega_1\leq\xi\leq\omega_2$,  every probability measure on
 $Ba^<[\xi]$ can be coded as a subset of $\xi$. Hence from \ref{shelah} we conclude the following.

\begin{corollary} Assume that Axiom \ref{ax2} holds and let $S$ be a stationary subset of $\omega_2$ such that $\cf(\alpha)=\omega$ for $\alpha\in S$. Then
there is a sequence $\langle \nu_\xi: \xi\in S\rangle$ where every $\nu_\xi$ is finitely additive probability measure
on $Ba^<[\xi]$ such that whenever $\nu$ is a finitely additive probability measure on $Ba[\omega_2]$ then
$\nu_{|Ba^<[\xi]}=\nu_\xi$ for stationary many $\xi\in S$.
\end{corollary}

\begin{xca}\label{construction}
Assume \ref{ax1} and \ref{ax2}. Fix a set $S\sub\omega_2$ as in \ref{ax1} and a $\diamondsuit_S$-sequence $\langle \nu_\xi: \xi\in S\rangle$ as in \ref{shelah}.
We shall define inductively a sequence $\langle \fA_\xi: \xi<\omega_2\rangle$ of algebras with the following properties

\begin{enumerate}[(i)]
\item $\CC[\xi]\sub \fA_\xi\sub Ba^<[\xi]$ for every $\xi<\omega_2$;
\item $\fA_\xi=\bigcup_{\alpha<\xi}\fA_\xi$ for every limit ordinal $\xi<\omega_2$;
\item whenever $\alpha\in\omega_2\sm S$ and $\alpha<\beta<\omega_2$ then every strictly positive measure on
$\fA_\alpha$ can be extended to a strictly positive measure on $\fA_\beta$;
\item for every $\xi\in S$, ${\nu_\xi}_{ | \fA_\xi}$ cannot be extended to a strictly positive measure on $\fA_{\xi+1}$.
\end{enumerate}
\end{xca}

For a limit ordinal $\xi$  we define $\fA_\xi$ by  \ref{construction}(ii); clearly (i) holds and  it is easy to check that property (iii) is preserved.

Given $\xi\notin S$ and $\fA_\xi$, we let $\fA_{\xi+1}$ be the algebra generated by $\fA_\xi$ and the set $C_\xi^0$.
Then (iii) is preserved by Lemma \ref{em:3}.

Finally,  consider $\xi\in S$. Then ${\rm cf}(\xi)=\omega$ so we may pick increasing sequence $(\alpha_n)_n$ cofinal  in
$\xi$. Then, by
inductive assumption, $\fA_\xi=\bigcup_n \fA_{\alpha_n}$. We define $\fA_{\xi+1}$ applying Proposition \ref{em:4}
(with $\nu=\nu_\xi$).

\begin{theorem}
Assume Axiom \ref{ax1} and \ref{ax2}, and let $\fA_\xi$ be the algebras given by \ref{construction}.
Then the  algebra $\fA=\bigcup_{\xi<\omega_2}\fA_\xi$ (of cardinality $\con^+$) is not in \spm\ but
$\fB\in\spm$ for  every subalgebra $\fB$ of $\fA$ of cardinality at most  $\con$.
\end{theorem}

\begin{proof}
Consider any probability measure $\mu$ on $\fA$. Let $\nu$ be any extension of $\mu$ to a probability measure
on $Ba[\omega_2]$. Then, by Theorem \ref{shelah},  $\nu_{|Ba^<[\xi]}=\nu_\xi$ for some $\xi\in S$.
It follows from \ref{construction}(iv) that $\mu$ is not strictly positive on $\fA_{\xi+1}$.
Hence $\fA\notin\spm$.

Let $\fB\sub\fA$ be an algebra with $|\fB|\le\omega_1=\con$; then $\fB\sub \fA_\xi$ for some $\xi<\omega_2$.
Therefore, to complete the proof it is enough to check that $\fA_\xi \in\spm$ for every  ordinal $\xi<\omega_2$ of cofinality $\omega_1$.

Let us fix  $\xi<\omega_2$ of cofinality $\omega_1$. Then $S\cap\xi$ is not stationary in $\xi$ so there is a set $F\sub\xi\sm S$
which is closed and unbounded in $\xi$. Then, using  \ref{construction}(iii),
we may define by induction on $\alpha\in F$ strictly positive measures $\mu_\alpha$ on $\fA_\alpha$ so that
$\mu_\alpha$ extends $\mu_\beta$ whenever $\beta,\alpha\in F$ and $\beta<\alpha$.
Now the common extension of those measures is strictly positive on  $\fA_\xi=\bigcup_{\alpha\in F}\fA_\alpha$.
Thus $\fA_\xi\in\spm$, and the proof is complete.
\end{proof}

\end{document}